%% file: framing_abstract.tex
\documentclass[12pt,a4paper]{article}
\usepackage{amsmath}
\usepackage{amssymb}
\usepackage{amsthm}
\usepackage{fullpage}
\usepackage{graphicx}
\usepackage{calc}
\usepackage[top=1.2in,bottom=1.2in,left=.9in,right=.9in]{geometry}
\usepackage[colorlinks=true,citecolor=red!65!white,linkcolor=red!65!blue,urlcolor=green!65!black,breaklinks=true]{hyperref}
\usepackage[capitalise,nameinlink,noabbrev]{cleveref}
\usepackage{booktabs}
\usepackage{multirow}
\usepackage{color}
\usepackage[sc]{mathpazo}
\usepackage[T1]{fontenc}
\usepackage[margin=10pt,font=small]{caption}
\usepackage{enumitem}

\usepackage{tikz}
\usetikzlibrary{arrows.meta}

\newtheorem{theorem}{Theorem}[section] 
\newtheorem{proposition}[theorem]{Proposition} 
\newtheorem{corollary}[theorem]{Corollary} 
\newtheorem{lemma}[theorem]{Lemma} 

\theoremstyle{definition}
\newtheorem{definition}[theorem]{Definition}
\newtheorem{remark}[theorem]{Remark} 
\newtheorem{example}[theorem]{Example}
\newtheorem{question}{Question}

\usepackage[backend=bibtex,giveninits,block=space,sortcites=true,maxbibnames=99,style=alphabetic]{biblatex}
\usepackage{doi}
\renewbibmacro{in:}{}
\DeclareFieldFormat[article]{volume}{\mkbibbold{#1}}
\DeclareFieldFormat[article]{number}{(#1)}
\DeclareFieldFormat{doi}{\href{https://dx.doi.org/#1}{\textsc{doi}}}
\AtEveryBibitem{\clearfield{issn}\clearfield{isbn}\clearlist{language}}
\renewbibmacro*{volume+number+eid}{%
  \printfield{volume}%
  \printfield{number}%
  \setunit{\bibeidpunct}%
  \printfield{eid}}
\def\abx@missing@entry#1{\abx@missing{#1??}}

\newcommand{\blue}{blue!65!white}
\newcommand{\mydef}[1]{\textcolor{\blue}{{\bf #1}}} 

\newcommand{\G}{\Gamma} 
\newcommand{\F}{F} 
\renewcommand{\c}{{\bf{c}}} 
\renewcommand{\d}{{\bf{d}}} 
\newcommand{\s}{\sigma} 
\newcommand{\sbar}{\overline\sigma} 
\renewcommand{\t}{\tau} 
\newcommand{\tbar}{\overline\tau} 
\newcommand{\up}{\uparrow} 
\newcommand{\down}{\downarrow} 
\renewcommand{\emph}{\textbf}

\newcommand{\inn}{\text{in}}
\newcommand{\out}{\text{out}}

\newcommand{\CCL}{\mathop{CC}_L} 
\newcommand{\CCR}{\mathop{CC}_R} 
\DeclareMathOperator{\DownBricks}{DownBricks}
\DeclareMathOperator{\UpBricks}{UpBricks}

\usepackage{authblk}

\author[1]{Jonah Berggren\thanks{\href{mailto:jrberggren@uky.edu}{jrberggren@uky.edu}. Jonah Berggren was partially supported by NSF grant DMS-2054255.}}
\author[2, 3]{Clément Chenevière\thanks{\href{mailto:clement.cheneviere@univ-eiffel.fr}{clement.cheneviere@univ-eiffel.fr}. Clément Chenevière was partially supported by the ANR--FWF project PAGCAP (ANR-21-CE48-0020).}}
\affil[1]{\small \itshape  University of Kentucky, Lexington, KY, United States}
\affil[2]{\small \itshape Université Paris-Saclay, CNRS, LISN, Orsay, France}
\affil[3]{\small \itshape Université Gustave Eiffel, LIGM, Champs-sur-Marne, France}

\title{Extended abstract:\\ Canonical join complex and cubical coordinates for all framing lattices}
\begin{document}

\maketitle

\abstract{%
	This document is an extended abstract for two articles in preparation.

	Recently, framing lattices were introduced to generalize many classical lattices such as the Tamari lattice and the weak order on $\mathfrak{S}_n$.
	We define bricks and brick cliques as a combinatorial model for join-irreducible elements and canonical join representations in all framing lattices, generalizing noncrossing arc diagrams of (Reading, 2015) for the weak order on $\mathfrak{S}_n$.
	Our model captures the natural bijection between join and meet canonical representations, as well as duality upon reflections of the framed graph.
	The proof is bijective, with an explicit reconstruction algorithm in two steps. A useful intermediate construction in our bijective proof is our new definition of cornered cliques.
	These enable us to define cubical coordinates on a framing lattice, generalizing bracket vectors for the Tamari lattice and providing an efficient comparison criterion.
}

\newpage

\section{Introduction}

Flow polytopes model the space of \emph{flows} on a directed acyclic graph, and are a fundamental object of combinatorial optimization with relations to many fields such as representation theory~\cite{WV,WIWT}, Grothendieck polynomials~\cite{LMD}, and algebraic geometry~\cite{EM}.
Danilov, Karzanov, and Koshevoy~\cite{DKK} defined special \emph{framing triangulations} of flow polytopes given a choice of a local order on all edges of internal vertices, called a framing of the graph. Simplices of the triangulation indexed combinatorially by \emph{maximal cliques} of the framed graph -- i.e., collections of routes which are pairwise noncrossing with respect to the framing. The maximal cliques (hence, the maximal simplices of these framing triangulations) were given the structure of a \emph{framing lattice} by von Bell and Ceballos~\cite{vBC}. Framing lattices are polygonal, semidistributive, and congruence uniform~\cite{vBC}, and they serve as a common generalization of many classical lattices such as the weak order on the symmetric group~\cite{Y2} and the Tamari lattice~\cite{Y1}.

It follows from semidistributivity that a framing lattice admits a \emph{canonical join complex}, which is the simplicial complex of canonical join representations of elements of the lattice~\cite{Barnard}. For some semidistributive lattices, such as the weak order on the symmetric group and the Tamari lattice, Reading introduced noncrossing arc diagrams as a combinatorial model for the canonical join complex~\cite{Reading}. In this abstract, we generalize the model of Reading to all framing lattices by introducing collections of noncrossing \emph{bricks}, henceforth referred to as \emph{brick cliques}, as a combinatorial model for canonical join representations. 
Covering relations in the framing lattice are naturally labeled by bricks, and the collection of labels of lower covers of a maximal clique forms its canonical join representation. 
Conversely, we build an explicit algorithm to reconstruct the maximal clique associated to a brick clique.

Moreover, we introduce \emph{left-cornered cliques} (dually, right-cornered cliques) as an intermediate object in our bijection between brick cliques and maximal cliques.
Hence, elements of the framing lattice may be equivalently described by maximal cliques, left-cornered cliques, or brick cliques. 
Through a (co)rank map on the left-cornered clique representations, we obtain \emph{cubical coordinates} on the framing lattice which are analogues of bracket vectors for the Tamari lattice~\cite{HT}, as comparison in the framing lattice translates to componentwise comparison in these coordinates. In the process, we attach to each maximal clique of routes a set of left-clockwise bricks, which play a role analogue to inversions for the weak order. 

The main objects and bijections involved in this abstract are summarized in ~\cref{fig:results_scheme}. We only focus on the bottom left quarter of the diagram, as the other parts follow by left-right and up-down reflections of the framed graph (each of which dualize the framing lattice). \cref{fig:intro} displays an example of the framing lattice on a small framed graph drawn using its maximal cliques (left), left-cornered cliques (middle), and brick cliques (right), with the embedding reflecting our cubical coordinates.

\begin{figure}
    \centering
    \scalebox{0.8}{
        \begin{tikzpicture}[scale=1.16]
            \node[draw,inner sep=1ex, align=center, text width=9em] (MaxCliques) at (0,0) {Maximal cliques of routes};
            \node[draw,inner sep=1ex, align=center, text width=9em] (LeftCornered) at (-7,0) {Left-cornered cliques};
            \node[draw,inner sep=1ex, align=center, text width=9em] (RightCornered) at (7,0) {Right-cornered cliques};
            \node[draw,inner sep=1ex, align=center, text width=9em] (UpBricks) at (0,4) {Brick cliques (canonical meet representation)};
            \node[draw,inner sep=1ex, align=center, text width=9em] (DownBricks) at (0,-4) {Brick cliques (canonical join representation)};
            \node[draw,inner sep=1ex, align=center, text width=9em] (LeftCoords) at (-7,-4) {Left cubical\\coordinates $\CCL(\Delta)$};
            \node[draw,inner sep=1ex, align=center, text width=9em] (RightCoords) at (7,-4) {Right cubical\\coordinates $\CCR(\Delta)$};

            \draw[-{>[length=2mm, width=3mm]}] (MaxCliques) to node[midway, fill=white] {Up labels} (UpBricks);
            \draw[-{>[length=2mm, width=3mm]}] (MaxCliques) to node[midway, fill=white, align=center] {Down labels \\ (\cref{thm:bijection_bricks})} (DownBricks);

            \draw[-{>[length=2mm, width=3mm]}] (MaxCliques) to node[midway, above] {Left-cornering map} node[midway, below] {$\Sigma_L$ (\cref{prop:left-cornering-inverse})} (LeftCornered);
            \draw[-{>[length=2mm, width=3mm]}] (MaxCliques) to node[midway, above] {Right-cornering} node[midway, below] {map $\Sigma_R$} (RightCornered);

            \draw[-{>[length=2mm, width=3mm]}, dashed, bend right=25] (UpBricks.west) to node[midway, above, sloped] {Reconstruction $\Phi^*_L$ (step 1)} (LeftCornered.north);
            \draw[-{>[length=2mm, width=3mm]}, dashed, bend left=25] (UpBricks.east) to node[midway, above, sloped] {Reconstruction $\Phi^*_R$ (step 1)} (RightCornered.north);
            \draw[-{>[length=2mm, width=3mm]}, dashed, bend left=25] (DownBricks.west) to node[midway, below, sloped] {Reconstruction $\Phi_L$ (step 1)} node[midway, above, sloped] {(\cref{def:first_reconstruction})} (LeftCornered.south);
            \draw[-{>[length=2mm, width=3mm]}, dashed, bend right=25] (DownBricks.east) to node[midway, below, sloped] {Reconstruction $\Phi_R$ (step 1)} (RightCornered.south);

            \draw[-{>[length=2mm, width=3mm]}, dashed, bend left=25] (LeftCornered.45) to node[midway, above, sloped] {Reconstruction $\Psi^*_L$ (step 2)} (MaxCliques.135);
            \draw[-{>[length=2mm, width=3mm]}, dashed, bend right=25] (RightCornered.135) to node[midway, above, sloped] {Reconstruction $\Psi^*_R$ (step 2)} (MaxCliques.45);
            \draw[-{>[length=2mm, width=3mm]}, dashed, bend right=25] (LeftCornered.-45) to node[midway, below, sloped] {Reconstruction $\Psi_L$ (step 2)} node[midway, above, sloped]{(\cref{def:second_reconstruction})} (MaxCliques.-135);
            \draw[-{>[length=2mm, width=3mm]}, dashed, bend left=25] (RightCornered.-135) to node[midway, below, sloped] {Reconstruction $\Psi_R$ (step 2)} (MaxCliques.-45);

            \draw[-{>[length=2mm, width=3mm]}] (LeftCornered) to node[midway, left] {Corank} (LeftCoords);
            \draw[-{>[length=2mm, width=3mm]}] (RightCornered) to node[midway, right] {Corank} (RightCoords);
        \end{tikzpicture}}
    \caption{Objects and bijections involved in this article. Left-right and up-down reflections of the framed graph explain left-right and up-down reflections of the diagram.}
    \label{fig:results_scheme}
\end{figure}
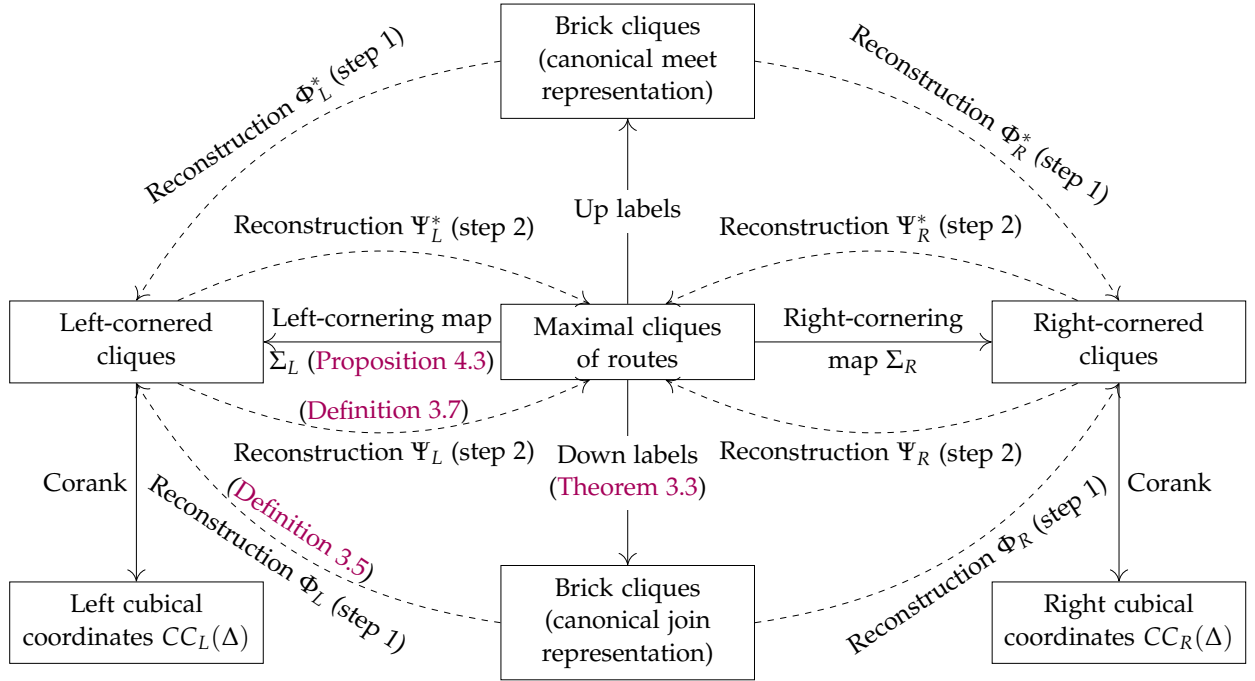

\begin{figure}
	\centering
	\def\svgscale{.25}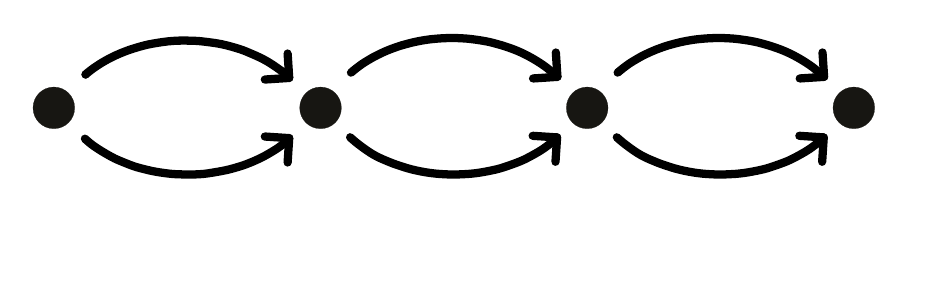

	\def\svgscale{.28}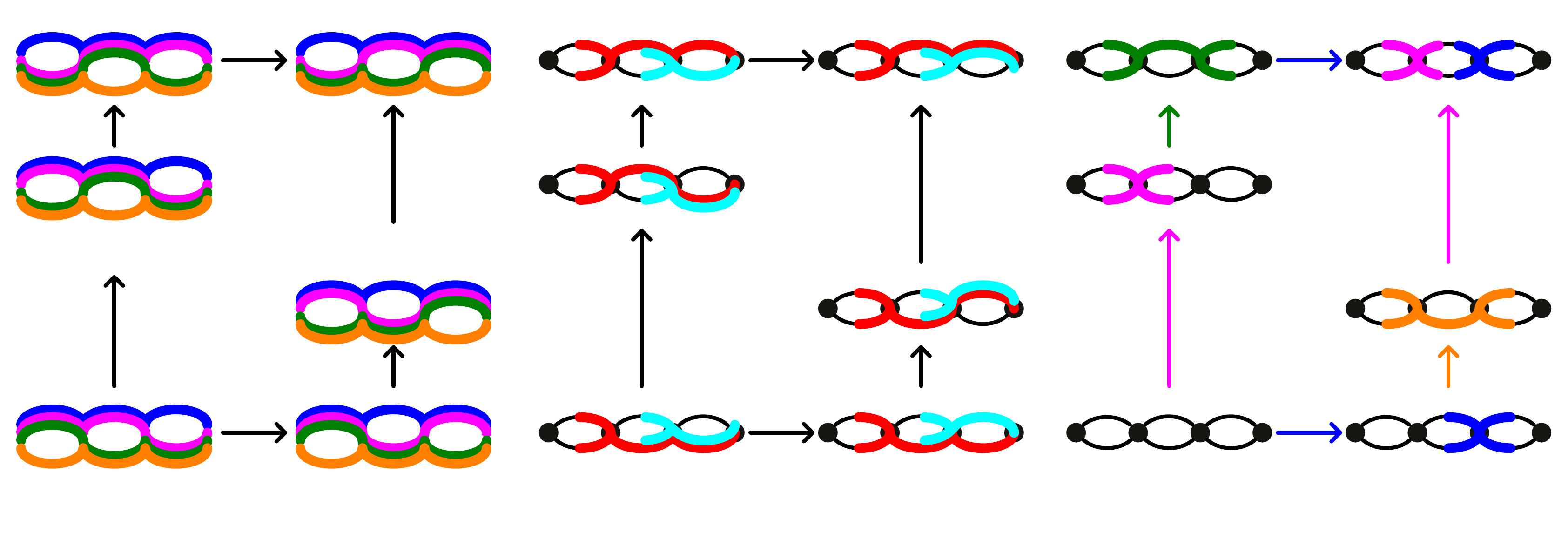
	\caption{The caracol graph of size $3$. The corresponding framing lattice is isomorphic to the weak order on $\mathfrak{S}_3$, displayed with left cubical coordinates on maximal cliques of routes (left), left-cornered cliques (middle) and brick cliques (right). On the right, bricks each have a different color and covering relations are colored accordingly.
	}
	\label{fig:intro}
\end{figure}

\section{Preliminaries on framing lattices}

We start by recalling some background on framing lattices, and we refer to~\cite{vBC} and the references therein for a more detailed account.
A \mydef{flow graph} $G$ is a directed acyclic graph with vertex set $V(G) = [n]$ and a multiset $E(G)$ of edges, all directed from a smaller to a larger vertex. A path from a source to a sink is called a \mydef{route}.

A \mydef{framing} $\F$ of $G$ is a choice at each internal vertex $v$ of total orders on incoming and on outgoing edges at $v$, denoted respectively $<_{\inn(v), \F}$ and $<_{\out(v), \F}$. The pair $\G = (G, \F)$ is called a \mydef{framed graph}. An example of framed graph is displayed at the top of~\cref{fig:intro}.
Following~\cite{vBC}, we draw a framed graph so that incoming (resp. outgoing) edges are increasing in the framing from top to bottom, in which case we may omit the framing.

When two routes share an internal vertex $v$, we can view their suffix after $v$ as words on edges, and use the framing on outgoing edges to compare them in the lexicographic order. This provides a preorder on all routes passing through $v$ called the \mydef{outgoing preorder at $v$} and denoted $\leq^+_{v, \F}$. The \mydef{incoming preorder at $v$} $\leq^-_{v, \F}$ is defined similarly by comparing prefixes.
We can also define an \mydef{outgoing order}~$\leq^+_\F$ as an order on all routes, where two routes starting with different source edges are incomparable, and compared in the corresponding outgoing preorder otherwise. An \mydef{incoming order} $\leq^-_\F$ is defined similarly, and we often omit $\F$ from the notation when it is clear from the context.

If two routes $p$ and $q$ share an internal vertex $v$, $p$ is \mydef{clockwise} to $q$ at $v$ if $p <^-_v q$ and $p >^+_v q$. Two routes are \mydef{coherent at $v$} if neither is clockwise to the other at $v$, and \mydef{coherent} if coherent at every internal vertex $v$.
If $p$ and $q$ are incoherent at $v$, their maximal common subpath containing $v$ is called a \mydef{subroute of conflict}. With the drawing conventions, incoherences appear visually as crossings.

A \mydef{maximal clique} $\Delta$ is a maximal collection of pairwise coherent routes. All maximal cliques have the same number of routes, namely \mbox{$\lvert E(G) \rvert - \lvert V(G) \rvert + s(G) + t(G)$}, where $s(G)$ (resp. $t(G)$) is the number of sources (resp. sinks) of $G$. Two cliques are adjacent when they share all but one route.

\begin{proposition}[{\cite[Lemma 2.3]{vBC}}]\label{prop:adjacent_cliques}
    Let $\Delta$ and $\Delta'$ be two adjacent maximal cliques, and $r$ and $r'$ the routes of $\Delta \setminus \Delta'$ and $\Delta' \setminus \Delta$. The routes $r$ and $r'$ have a unique subroute of conflict $\s$.
	Moreover, writing $r = r_1 \s r_2$ and $r' = r'_1 \s r'_2$, the routes $r_1 \s r'_2$ and $r'_1 \s r_2$ are both in $\Delta \cap \Delta'$.
\end{proposition}

We can build a directed graph on maximal cliques with edges between adjacent cliques. Using the notations of~\cref{prop:adjacent_cliques}, the edge is oriented from $\Delta$ to $\Delta'$ if $r$ is clockwise to $r'$ at their subroute of conflict, and replacing $r$ with $r'$ is called a \mydef{counterclockwise rotation}. von Bell and Ceballos proved that this directed graph is the Hasse diagram of a semidistributive lattice called the \mydef{framing lattice} $\mathcal L_\Gamma$~\cite{vBC}.

The \mydef{up-down reflection} of the framed graph $\G$ is the framed graph $\G^* = (G, F^*)$, where the underlying graph is unchanged but the outgoing and incoming orders at each vertex are reversed. Similarly, the \mydef{left-right reflection} of $\G$ is the framed graph $\G^T = (G^T, F^T)$, where $G^T$ is the graph obtained from $G$ by reversing all edges, and the incoming and outgoing orders at each vertex have been exchanged. Up-down and left-right reflections correspond to the horizontal and vertical reflections of our pictures, respectively. Both of these operations dualize the framing lattice.

\begin{example}\label{ex:weak_order}
	The weak order on the symmetric group $\mathfrak{S}_n$ can be realized as a framing lattice on the oruga graph, which has vertex set $[0, n]$ and two edges $(i-1,i)$ for all $i$. The framing is such that the outgoing order at vertex $i$ matches the incoming order at vertex $i+1$.
	Routes on the oruga graph can be seen as binary words where going down is a $0$ and going up is a $1$. Given a permutation with one-line notation $i_1 \dots i_n$, we associate the collection $\{p_0, \dots, p_k\}$ of binary words, where $p_k$ has $k$ ones, at positions $i_1, \dots, i_k$.
	An example for $n = 3$ can be seen in~\cref{fig:intro}, where on the left, the permutations are displayed as cliques of coherent routes and with their one-line notation.
\end{example}

\section{Bricks and canonical join complex}

Our first contribution is the definition of left and right corners, which we will allow to be the start and end, respectively, of generalized paths of $\G$. 
A \mydef{left corner} of $\G$ is a triple $\c=(\c^\bullet,\c^\down,\c^\up)$ where $\c^\bullet$ is an internal vertex of $G$ and $\c^\down,\c^\up$ are two incoming edges of $\c^\bullet$, consecutive in the framing. A \mydef{right corner} $\c=(\c^\bullet,\c^\down,\c^\up)$ is defined similarly with consecutive outgoing edges of $\c^\bullet$. 
Note that $\c^\up$ is larger than $\c^\down$ in the framing, so it is drawn below $\c^\down$, and arrows in superscript point towards the corner.

We generalize paths by letting paths of $G$ start or end with respectively left or right corners. 
More formally, a \mydef{generalized path} $\sbar$ of $\G$ is a composition $L(\sbar) \s R(\sbar)$ where $\s$ is a path from $v$ to $v'$, $L(\sbar)$ is either $v$ or a left corner with $\left(L(\sbar)\right)^\bullet = v$, and $R(\sbar)$ is either $v'$ or a right corner with $\left(R(\sbar)\right)^\bullet = v'$. When $L(\sbar)$ is a left corner and $R(\sbar)$ is a vertex, we call $\sbar$ a \mydef{left-cornered path}, and a \mydef{left-cornered route} if furthermore $R(\sbar)$ is a sink. Right-cornered paths and routes are defined similarly. We call $\sbar$ a \mydef{brick} if $L(\sbar)$ and $R(\sbar)$ are  corners. Routes, cornered routes and bricks are called \mydef{generalized routes}.

We say that a generalized path $\sbar$ \mydef{passes through} a vertex $v$ if $v$ is a vertex of~$\s$. 
We extend the framing order $\leq_{\inn(v), F}$ to a total order on incoming edges and left corners of $v$ by setting $\c^\down <_{\inn(v), F} \c <_{\inn(v), F} \c^\up$, and similarly for  $\leq_{\out(v), F}$.
We also extend naturally the outgoing and incoming preorders at $v$ to all generalized routes passing through $v$. 

If $\sbar$ and $\tbar$ are two generalized routes passing through $v$, we say that $\sbar$ is \mydef{clockwise} to $\tbar$ at $v$ if $\sbar <^-_v \tbar$ and $\sbar >^+_v \tbar$. They are \mydef{weakly coherent} (or {noncrossing}) if neither is clockwise to the other at any vertex, and they are \mydef{coherent} if they additionally do not share a left or a right corner.
Two generalized routes are incoherent when they have a crossing, as pictured on~\cref{fig:coherence}.
A \mydef{left-cornered clique} is a maximal collection of pairwise coherent left-cornered routes.
We easily see that left-cornered cliques have one left-cornered route starting at each left corner.
\mydef{Right-cornered cliques} are defined similarly. A \mydef{brick clique} is a (possibly empty) collection of pairwise coherent bricks. 
The \emph{brick complex} is the simplicial complex of brick cliques.
In the rest of this section, we prove that the brick complex is isomorphic to the canonical join complex of $\mathcal{L}_\G$. We first recall a few theoretical results on semidistributive lattices, for which we refer to~\cite{Barnard}.

\begin{figure}
	\centering
	\def\svgscale{.28}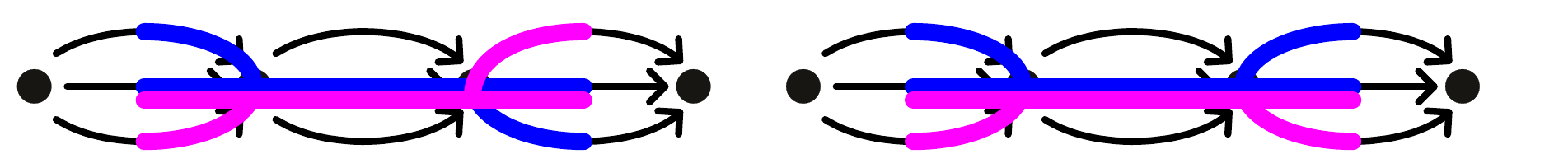
	\caption{The bricks on the left are incoherent, while those on the right are coherent.}
	\label{fig:coherence}
\end{figure}

An element of a finite lattice is called \mydef{join-irreducible} (resp. meet-irreducible) if it covers (resp. is covered by) a unique element. 
A lattice is distributive if the join and meet operations distribute over each other. Semidistributivity is a weakening of distributivity.

A key consequence of semidistributivity is that every element admits a ``minimal'' expression as the join of join-irreducible elements, called its \mydef{canonical join representation}.
There is moreover a canonical way to label covering relations of the lattice with join irreducible elements in such a way that the canonical join representation of an element is the set of labels of its lower covers.
The \mydef{canonical join complex} $\mathop{CJC}(L)$ of $L$ is the simplicial complex whose faces are canonical join representations.
Canonical meet representation and the canonical meet complex $\mathop{CMC}(L)$ are defined analogously, and for a semidistributive lattice 
$\mathop{CJC}(L)$ and $\mathop{CMC}(L)$ are isomorphic and flag~\cite{Barnard}. 

We first extend~\cref{prop:comparison_CvB} to label covering relations of the $\mathcal{L}_\G$ with bricks, and define two maps $\DownBricks$ and $\UpBricks$ from maximal cliques to brick cliques.

\begin{lemma}\label{lem:cov_bricks}
	In the situation of~\cref{prop:adjacent_cliques}, the last (resp. first) edges of paths $r_1$ and $r'_1$ (resp. $r_2$ and $r'_2$) are consecutive in the framing and define a left corner $\c_1$ (resp. right corner $\c_2$). We label the covering relation with the brick $\c_1 \s \c_2$.
\end{lemma}

\begin{lemma}\label{lem:coherent_lower_bricks}
	Two bricks labeling lower (resp. upper) covering relations of a clique $\Delta$ are coherent.
\end{lemma}

\begin{theorem}\label{thm:bijection_bricks}
	The map $\DownBricks$ that sends a maximal clique $\Delta$ to the set of bricks labeling its lower covering relations is a bijection from maximal cliques of routes to brick cliques.

	The map $\UpBricks$ defined similarly with upper covering relations of $\Delta$ is also a bijection.
\end{theorem}

\begin{corollary}\label{cor:canonical_join_complex}
	Join-irreducible (resp. meet-irreducible) elements of $\mathcal{L}_\G$ are in bijection with bricks. The brick complex is isomorphic to $\mathop{CJC}(\mathcal{L}_\G)$ (resp. $\mathop{CMC}(\mathcal{L}_\G)$) via this bijection. 
\end{corollary}

In order to prove~\cref{thm:bijection_bricks}, we define the inverse of the map $\DownBricks$ through an explicit reconstruction algorithm in two steps.
We define a first map $\Phi_L$ from brick cliques to left-cornered cliques, then a second one $\Psi_L$ to maximal cliques. The left-right reflection of $\G$ defines a `right' version $\Phi_R$ and $\Psi_R$ of the reconstruction maps, and the up-down reflection give a `dual' version of the reconstruction, inverse to $\DownBricks$.

\begin{definition}\label{def:first_reconstruction}
	Let $T$ be a brick clique and $\c$ a left corner of $\G$.
	We define the \mydef{reconstruction} $\Phi_L(T, \c)$ by growing a left-cornered path starting with $\sbar_0 = \c$.
	Iteratively define $\sbar_{i+1}=\sbar_i e_{i+1}$, where $e_{i+1}$ is the largest outgoing edge of $v_i = R(\sbar_i)$ such that for all bricks $\tbar$ of $T$ with $v_i \in \tbar$, if $\sbar_{i} \leq^-_{v_i} \tbar$ then also $e_{i+1} \leq^+_{v_i} \tbar$.
	This process concludes with a left-cornered route $\sbar_k$, which we call $\Phi_L(T,\c)$.
\end{definition}

By construction, $\Phi_L(T, \c)$ is the largest left-cornered route in $\leq^+_\c$ that is weakly coherent with $T$, and with $\Phi_L(T, \c) >^-_{\c^*} \tbar_\c$ for the brick $\tbar_\c \in T$ with left corner~$\c$ if any.
The \mydef{first reconstruction map} $\Phi_L$ sends a brick clique $T$ to the collection $\{\Phi_L(T, \c)\}$, where $\c$ runs over all left corners of $\G$. 

\begin{example}\label{ex:first_reconstruction}
	An example is displayed on the bottom left of~\cref{fig:reconstruction}.
	From each corner, displayed on the left, its image under $\Phi_L$ is shown. 
	For corner $ab = (2, a, b)$, the reconstruction algorithm chooses edges $e, g, k$. Note for instance that edge $h$ was not picked as $h >^+_3 gh$. 
	For corner $bc$, note also that the path $fh$ also gives a left-cornered path coherent with the brick clique, but $fik$ produces a left-cornered route larger for $\leq^+_2$.
\end{example}

The second step of the reconstruction is very similar to the first, except that we begin with a sink edge or left-cornered route and iteratively pick the smallest possible incoming edge to extend our reconstructed path to the left.

\begin{definition}\label{def:second_reconstruction}
	Let $LT$ be a left-cornered clique and $\d$ a left corner or a sink edge of~$\G$. Let $\tbar_\d$ be the left-cornered route of $LT$ containing $\d$ if $\d$ is a corner.
	Let $p_0 = \d^\up \t_\d$ if~$\d$ is a corner and $p_0 = \d$ otherwise. Iteratively define $p_{i+1} = e_{i+1} p_i$, where $e_{i+1}$ is the smallest incoming edge of $v_i = L(p_i)$ such that for all left-cornered bricks $\tbar$ of $LT$ passing through $v_i$, if $p_i \geq^+_{v_i} \tbar$ then $e_{i+1} \geq^-_{v_i} \tbar$.
	This process concludes with a route $p_k$, which we denote $\Psi_L(LT, \d)$. 
\end{definition}

By construction, $\Psi_L(LT, \d)$ is the smallest route in $\leq^-$ which contains $p_0$ and is coherent with all $\tbar \in LT$. The \mydef{second reconstruction map} $\Psi_L$ sends a left-cornered clique $LT$ to the set $\{\Psi_L(LT, \d)\}$, where $\d$ runs over all left-corners and sink edges of $\G$.

\begin{example}\label{ex:second_reconstruction}
	An example of the second step of reconstruction algorithm $\Psi_L$ is displayed on the right of~\cref{fig:reconstruction}.
	Each left-cornered route provides a route, as well as each sink edge, above. The part coming from the left-cornered route or the edge is drawn in pink, whereas the reconstructed part is in orange.
	For corner $gi$, the reconstructed route uses for instance edge $b$ since edge $c$ is larger in the framing, and edge $a$ would produce a route incoherent with the left-cornered route associated to corner $ab$. 
\end{example}

\begin{figure}[ht]
	\centering
	\def\svgscale{.26}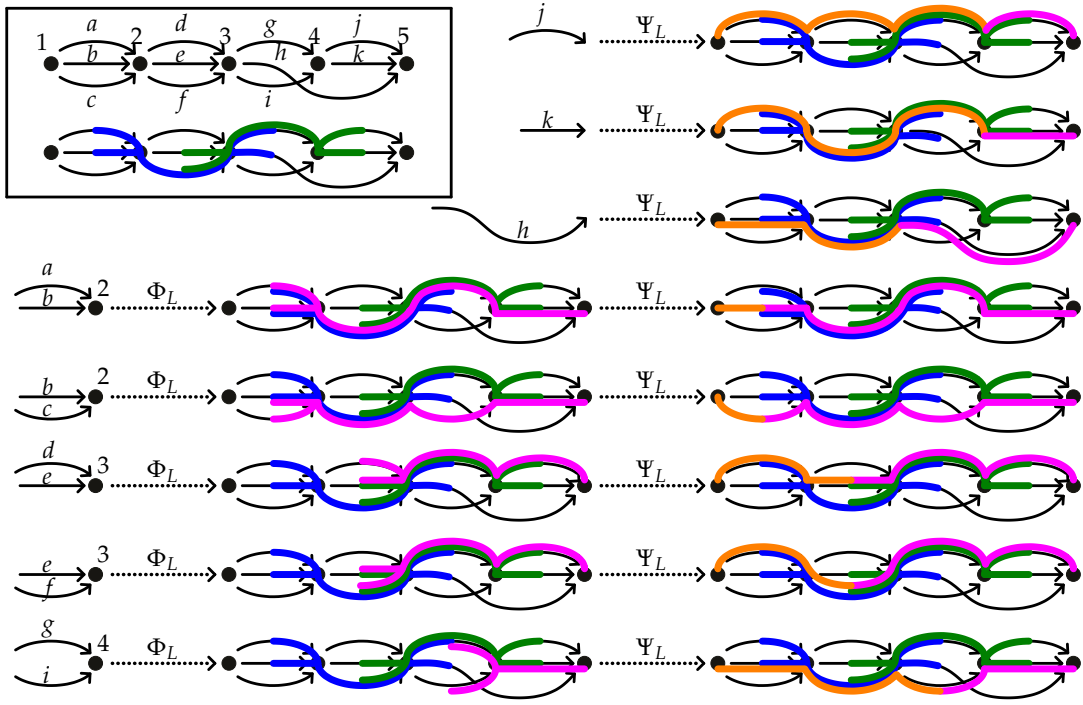
	\caption{An example of the reconstruction algorithm. Top left is displayed the framed graph which has five left corners, with a clique of two coherent bricks (green and blue). Below is the first step of reconstruction $\Phi_L$. The corner from which each left-cornered route (in pink) is grown is desplayed on the left. On the right is the second step of reconstruction $\Psi_L$, with five routes reconstructed from a left-cornered route, and three routes reconstructed from a sink edge. The reconstructed part is in orange.}
	\label{fig:reconstruction}
\end{figure}

We now sketch the proof of~\cref{thm:bijection_bricks}, focusing on the map $\DownBricks$. 

\begin{proof}[Proof of~\cref{thm:bijection_bricks}]
	Firstly, if $\Delta$ is a maximal clique then $\DownBricks(\Delta)$ is a brick clique by~\cref{lem:coherent_lower_bricks}. If $T$ is a brick clique, then $\Psi_L \circ \Phi_L(T)$ is a collection of pairwise compatible routes. Moreover, since $\lvert E(G) \rvert - \lvert V(G) \rvert + s(G) + t(G)$ is equal to the number of left corners and sink edges, $\Psi_L \circ \Phi_L(T)$ is a maximal clique.

	By construction, the routes of $\Psi_L \circ \Phi_L(T)$ which can be rotated clockwise are exactly those reconstructed from left corners of bricks of $T$. Moreover, these bricks of~$T$ are exactly those labeling these lower covers, so that \mbox{$\DownBricks(\Psi_L \circ \Phi_L(T)) = T$}.

	Given a brick clique $T$, the route reconstructed from a left corner (resp. sink edge)~$\d$ is by construction the most clockwise route containing $\d^\up$ (resp. $\d$), not clockwise to any brick of $T$ but counterclockwise to the brick $\tbar_\d \in T$ with \mbox{$L(\tbar_\d) = \d$} if any. 
	For each route $r \in \Delta$, there is a left corner or sink edge $\d$ such that~$r$ satisfies the above property with respect to $\d$ and $\DownBricks(\Delta)$, hence we have $\Psi_L \circ \Phi_L(\DownBricks(\Delta)) = \Delta$. 
\end{proof}

We end this section by explaining how we recover noncrossing arc diagrams of Reading~\cite{Reading} on our running example of the weak order on $\mathfrak{S}_n$ (see~\cref{ex:weak_order}).

\begin{example}\label{ex:noncrossing_arcs}
	Displaying numbers from $1$ to $n$ in order on a line, an arc is a curve $\alpha$ that starts at an integer $i$ and ends at an integer $j > i$, going either above or below each integer $k \in \left(i, j\right)$.
	We map such an arc to the brick $\sbar_\alpha$ such that $L(\sbar_\alpha) = i$, $R(\sbar_\alpha) = j-1$ and $\s_\alpha$ uses the top edge between $k-1$ and $k$ if and only if $\alpha$ goes above $k$.
	A noncrossing arc diagram is a collection of arcs which do not cross and do not share a starting point or an ending point. These conditions immediately translate through the above map to being coherent bricks. An example for $n=3$ is displayed on the right of~\cref{fig:intro}.
\end{example}

\section{Cornered routes and cubical coordinates}

We now further explore left-cornered cliques, which are the intermediate object appearing in the reconstruction algorithm.
We define explicitely the inverse map of the second reconstruction step.
Taking coranks in the outgoing order, we define left cubical coordinates on the framing lattices, named as such because they embed the Hasse diagram with edges in the direction of the cube.
In the process, we introduce left-incoherent bricks, which play a role analogue to left inversions for the weak order on the symmetric group.
We prove that comparison in the framing lattice corresponds to componentwise comparison of cubical coordinates, generalizing bracket vectors for the Tamari lattice and providing an efficient comparison criterion. 

We start by defining upper and lower cornering routes of a corner. Since all routes within a clique are coherent, its subset of routes going through a vertex is totally ordered. 
\begin{definition}\label{def:cornering-routes}
	Let $\c = (\c^\bullet,\c^\up,\c^\down)$ be a left corner of $\G$ and $\Delta$ a maximal clique of routes. The \mydef{upper cornering route} $r^\up_\c(\Delta)$ of $\Delta$ at corner $\c$ is the smallest route in $\Delta$ using the edge $\c^\up$. The \mydef{lower cornering route} $r^\down_\c(\Delta)$ is defined symmetrically. 
\end{definition}

\begin{lemma}\label{lem:cornering-routes}
	Let $\c$ be a left corner of $\G$ and $\Delta$ a maximal clique of routes. The routes $r^\down_\c(\Delta)$ and $r^\up_\c(\Delta)$ coincide after $\c$ and are consecutive within $\Delta$ in the incoming order.
	In particular, a route $p$ of $\Delta$ can not be the lower (resp. upper) cornering route of $\Delta$ at two different left corners. 
\end{lemma}

Since $r^\up_\c(\Delta)$ and $r^\down_\c(\Delta)$ coincide after $\c$, they define a unique left-cornered route with corner $\c$, which we call the \mydef{left-cornered route 
$\sbar_\c(\Delta)$
of $\Delta$ at corner $\c$}. We write $\Sigma_L(\Delta)$ for the collection of all left-cornered routes of $\Delta$. We omit the proof of the following result, which is very similar to the one of~\cref{thm:bijection_bricks}. 

\begin{proposition} \label{prop:left-cornering-inverse}
	The set $\Sigma_L(\Delta)$ forms a left-cornered clique. The \mydef{left-cornering map} $\Sigma_L$ is the inverse map of the second reconstruction step $\Psi_L$ defined in~\cref{def:second_reconstruction}. 
\end{proposition}

Recall that the outgoing order at $\c^\bullet$ is a total order on all left-cornered routes with left corner $\c$.
We define the \mydef{left cubical coordinate} $\CCL(\Delta, \c)$ of $\Delta$ at corner $\c$ as the corank of $\sbar_\c(\Delta)$ in the outgoing order.
The name is justified by the following theorem.

\begin{theorem}
	\label{thm:cubical-coordinates}
	Let $\Delta \lessdot \Delta'$ be a covering relation in the framing lattice $\mathcal L_\Gamma$.
	The left cubical coordinates of $\Delta$ and $\Delta'$ differ at a single corner $\c$, with $\CCL(\Delta,\c) < \CCL(\Delta',\c)$.
\end{theorem}

In order to prove~\cref{thm:cubical-coordinates}, we introduce the notion of left-clockwise bricks, as analogues of inversions for the symmetric group. A brick $\sbar$ is \mydef{left-clockwise} to a route $p$ if $\sbar$ is clockwise to $p$ at its left corner $L(\sbar)$, i.e. if we have $\sbar <_v^- p$ and $\sbar >_v^+ p$ for $v = L(\sbar)^\bullet$.
The brick is left-clockwise to a clique $\Delta$ if it is left-clockwise to a route of $\Delta$. We have the following lemmas.

\begin{lemma}\label{lem:left-incompatibility}
	For a brick $\sbar$ with left corner $\c$ and a maximal clique $\Delta$, the following are equivalent:
	\begin{enumerate}
		\item $\sbar$ is left-clockwise to $\Delta$;
		\item $\sbar$ is larger than the left-cornered route $\sbar_\c(\Delta)$ of $\Delta$ in the outgoing order at $\c^\bullet$;
		\item $\sbar$ is left-clockwise to the upper cornering route $r^\up_\c(\Delta)$ of $\Delta$ at corner $\c$. 
	\end{enumerate}
\end{lemma}

\begin{lemma}\label{lem:left-incompatibility-ideal}
	The set of bricks with left corner $\c$ that are left-clockwise to a clique $\Delta$ is an upper ideal in the outgoing order at $\c^\bullet$, and their number is equal to $\CCL(\Delta,\c)$.
\end{lemma} 

\begin{proof}[Proof of \cref{thm:cubical-coordinates}]
	Let $\sbar$ be the brick labeling a covering relation $\Delta \lessdot \Delta'$, with left corner $L(\sbar) = \c$.
	Let $r$ (resp. $r'$) be the route of $\Delta \setminus \Delta'$ (resp. $\Delta' \setminus \Delta$).
	The lower cornering route of $\Delta$ at $\c$ is $r = r^\down_{\c}(\Delta)$, and $\sbar$ is counterclockwise to $r$ but clockwise to $r'$, so by \cref{lem:left-incompatibility}, $\sbar$ is left-clockwise to $\Delta'$ but not $\Delta$. By \cref{lem:left-incompatibility-ideal}, we have $\CCL(\Delta, \c) < \CCL(\Delta', \c)$.

	Let $\tbar$ be a brick with left corner $\c' \neq \c$. 
	If $\tbar$ is left-clockwise to $\Delta$, then by \cref{lem:left-incompatibility}, it is incoherent with $p := r^\down_{\c'}(\Delta)$. Since $r = r^\down_{\c}(\Delta)$, $r \neq p$ by \cref{lem:cornering-routes}, so that $p \in \Delta'$ and $\tbar$ is left-clockwise to $\Delta'$.
	Conversely, reusing notations of~\cref{prop:adjacent_cliques}, $\Delta$ contains the two routes $r_1 \s r_2'$ and $r'_1 \s r_2$. 
	If $\tbar$ is left-clockwise to $r'$, it is left-clockwise to either $r_1 \s r_2'$ or $r'_1 \s r_2$, so $\tbar$ is left-clockwise to $\Delta$. 
	Hence, $\CCL(\Delta, \c') = \CCL(\Delta', \c')$ by \cref{lem:left-incompatibility-ideal}.
\end{proof}

The motivating example is that the Tamari lattice (and all type $A$ Cambrian lattices) can be realized as a framing lattice of the so-called caracol graph, 
which has vertex set $[0, n+2]$ and edges $(i, i+1)$ for all $i$, as well as edges $(0, i)$ and $(i, n+2)$ for $2 \leq i \leq n$. The framing that produces the Tamari lattice is the one where edges $(i, i+1)$ are always the smallest in the outgoing and incoming orders~\cite{Y1,vBC}.

\begin{example}\label{ex:cubical-coordinates-tamari}
	For the Tamari lattice on the framed caracol graph, 
	the left cubical coordinates recover precisely the classical bracket vectors of~\cite{HT}.
\end{example}

Not only do these cubical coordinates provide a nice embedding of framing lattices, but they also give an easy comparison criterion, as bracket vectors for the Tamari lattice.

\begin{proposition}[{\cite[Proposition 3.4]{vBC}}]\label{prop:comparison_CvB}
	Let $\Delta$ and $\Delta'$ be two maximal cliques of routes. We have $\Delta \not\leq \Delta'$ in the framing lattice if and only if a route of $\Delta'$ is clockwise to a route of $\Delta$ at a vertex $v$.
\end{proposition}

\begin{theorem}\label{thm:componentwise_comparison}
	Let $\Delta$ and $\Delta'$ be two maximal cliques of routes. We have $\Delta \leq \Delta'$ in the framing lattice if and only if $\CCL(\Delta) \leq \CCL(\Delta')$ componentwise.
\end{theorem}

\begin{proof}
	If $\Delta \leq \Delta'$ in $\mathcal L_\Gamma$, there is a sequence of covering relations from $\Delta$ to $\Delta'$, and by~\cref{thm:cubical-coordinates} we conclude that $\CCL(\Delta) \leq \CCL(\Delta')$ componentwise.
	
	Reciprocally, if  $\Delta \not\leq \Delta'$ in $\mathcal L_\Gamma$, by~\cref{prop:comparison_CvB} there is a route $r' \in \Delta'$ clockwise to a route $r \in \Delta$ at a vertex $v$. Taking $v$ minimal with this property, $r$ and $r'$ are not using the same incoming edge of $v$. Pick any left corner $\c$ of $v$ with $r' <^-_v \c <^-_v r$.
	Since $r'$ is clockwise to $r$ at $v$, we have $r <^+_v r'$. 
	By~\cref{def:cornering-routes}, $r \geq^+_v r^\up_\c(\Delta)$ and $r' \leq^+_v r^\down_\c(\Delta')$, so that $r^\up_\c(\Delta) <^+_v r^\down_\c(\Delta')$. Hence, $\CCL(\Delta, \c) > \CCL(\Delta', \c)$. 
\end{proof}

\begin{remark}\label{rem:inversions}
	Combining \cref{thm:componentwise_comparison,lem:left-incompatibility-ideal}, we obtain yet another description of comparison in the framing lattice as inclusion of sets of left-clockwise bricks, hence our claim that they play a role analogue to inversions for the weak order. However, the meet operation does not correspond to intersection of sets of left-clockwise bricks.
\end{remark}

\begin{example}\label{ex:cubical-coordinates-weak}
	Recall from~\cref{ex:weak_order} that the weak order on $\mathfrak{S}_n$ can be realized as a framing lattice on the oruga graph, 
	which has one left corner for each integer~\mbox{$i\in[n-1]$}.
	For a permutation $\pi \in \mathfrak{S}_n$, the left-cornered path $\sbar_i(\Delta_\pi)$ uses the top edge between $j$ and $j+1$ if and only if $(i, j+1)$ is an inversion of $\pi$. 
	The left cubical coordinate $\CCL(\Delta_\pi, i)$ is equal to $\sum_{(i, j) \in \operatorname{inv}(\pi)} 2^{n-j}$, where the sum runs over inversions $(i, j)$ of $\pi$ with $i < j$.

	The Lehmer code, which can be obtained by replacing $2^{n-j}$ by $1$ in the above formula, also produces a cubical embedding of the weak order in the sense of~\cref{thm:cubical-coordinates}, but componentwise comparison does not recover the weak order as in~\cref{thm:componentwise_comparison}.
\end{example}

In this section, we focused on left-cornered cliques, but of course all our constructions and results extend symmetrically to right-cornered cliques, allowing us to define right cubical coordinates, which give a different embedding of the framing lattice.
As a corollary of~\cref{thm:componentwise_comparison}, we obtain that the order dimension of a framing lattice is at most the number of left (resp. right) corners of $\G$---this bound is tight for the weak order.

We also remark that~\cref{prop:left-cornering-inverse} extends under up-down reflection.
Interestingly, the left-cornering map $\Sigma_L$ is invariant under up-down reflection.
Hence, $\Psi_L$ and $\Psi_L^*$ are inverse to the same, so they coincide as maps on maximal cliques, even though for a given corner $\c$, we always have $\Psi_L(LT, \c) \neq \Psi_L^*(LT, \c)$.  

\section{Concluding remarks}

We conclude this extended abstract with a few further directions of ongoing research.

\begin{question}
	Refine our description of the canonical join complex of a framing lattice into a model for the canonical complex, generalizing semi-noncrossing arc bidiagrams~\cite{DP}.
\end{question}

\begin{question}
	Study rowmotion on framing lattices using the map
		$\Psi_L \circ \Phi_L \circ \text{Upbricks}$.
\end{question}

\begin{question}
	Find a combinatorial description of the forcing order on the canonical join complex in terms of bricks, and use it to understand quotients of framing lattices.
\end{question}

\section*{Acknowledgements}
We are very grateful to Jean-Christophe Novelli, Wenjie Fang, and Éric Fusy for insightful discussions and comments on preliminary versions of this work.

\renewcommand{\emph}{\textit} 
\printbibliography

\end{document}

%% file: caracoldag.pdf_tex
\begingroup%
  \makeatletter%
  \providecommand\color[2][]{%
    \errmessage{(Inkscape) Color is used for the text in Inkscape, but the package 'color.sty' is not loaded}%
    \renewcommand\color[2][]{}%
  }%
  \providecommand\transparent[1]{%
    \errmessage{(Inkscape) Transparency is used (non-zero) for the text in Inkscape, but the package 'transparent.sty' is not loaded}%
    \renewcommand\transparent[1]{}%
  }%
  \providecommand\rotatebox[2]{#2}%
  \newcommand*\fsize{\dimexpr\f@size pt\relax}%
  \newcommand*\lineheight[1]{\fontsize{\fsize}{#1\fsize}\selectfont}%
  \ifx\svgwidth\undefined%
    \setlength{\unitlength}{449.17401123bp}%
    \ifx\svgscale\undefined%
      \relax%
    \else%
      \setlength{\unitlength}{\unitlength * \real{\svgscale}}%
    \fi%
  \else%
    \setlength{\unitlength}{\svgwidth}%
  \fi%
  \global\let\svgwidth\undefined%
  \global\let\svgscale\undefined%
  \makeatother%
  \begin{picture}(1,0.30031792)%
    \lineheight{1}%
    \setlength\tabcolsep{0pt}%
    \put(0.16249255,0.02736113){\color[rgb]{0,0,0}\makebox(0,0)[lt]{\lineheight{1.25}\smash{\begin{tabular}[t]{l}$\scriptstyle{a_2}$\end{tabular}}}}%
    \put(0.16249255,0.29051342){\color[rgb]{0,0,0}\makebox(0,0)[lt]{\lineheight{1.25}\smash{\begin{tabular}[t]{l}$\scriptstyle{a_1}$\end{tabular}}}}%
    \put(0.44746006,0.28858279){\color[rgb]{0,0,0}\makebox(0,0)[lt]{\lineheight{1.25}\smash{\begin{tabular}[t]{l}$\scriptstyle{b_1}$\end{tabular}}}}%
    \put(0.44745952,0.02736547){\color[rgb]{0,0,0}\makebox(0,0)[lt]{\lineheight{1.25}\smash{\begin{tabular}[t]{l}$\scriptstyle{b_2}$\end{tabular}}}}%
    \put(0.73242757,0.28858279){\color[rgb]{0,0,0}\makebox(0,0)[lt]{\lineheight{1.25}\smash{\begin{tabular}[t]{l}$\scriptstyle{c_1}$\end{tabular}}}}%
    \put(0.73242703,0.02736547){\color[rgb]{0,0,0}\makebox(0,0)[lt]{\lineheight{1.25}\smash{\begin{tabular}[t]{l}$\scriptstyle{c_2}$\end{tabular}}}}%
    \put(0.59755109,0.24702866){\color[rgb]{0,0,0}\makebox(0,0)[lt]{\lineheight{1.25}\smash{\begin{tabular}[t]{l}$\scriptstyle{3}$\end{tabular}}}}%
    \put(0.31258413,0.24702866){\color[rgb]{0,0,0}\makebox(0,0)[lt]{\lineheight{1.25}\smash{\begin{tabular}[t]{l}$\scriptstyle{2}$\end{tabular}}}}%
    \put(0.00980452,0.22921927){\color[rgb]{0,0,0}\makebox(0,0)[lt]{\lineheight{1.25}\smash{\begin{tabular}[t]{l}$\scriptstyle{1}$\end{tabular}}}}%
    \put(0.88251806,0.24702974){\color[rgb]{0,0,0}\makebox(0,0)[lt]{\lineheight{1.25}\smash{\begin{tabular}[t]{l}$\scriptstyle{4}$\end{tabular}}}}%
    \put(0,0){\includegraphics[width=\unitlength,page=1]{caracoldag.pdf}}%
  \end{picture}%
\endgroup%

%% file: trifecta.pdf_tex
\begingroup%
  \makeatletter%
  \providecommand\color[2][]{%
    \errmessage{(Inkscape) Color is used for the text in Inkscape, but the package 'color.sty' is not loaded}%
    \renewcommand\color[2][]{}%
  }%
  \providecommand\transparent[1]{%
    \errmessage{(Inkscape) Transparency is used (non-zero) for the text in Inkscape, but the package 'transparent.sty' is not loaded}%
    \renewcommand\transparent[1]{}%
  }%
  \providecommand\rotatebox[2]{#2}%
  \newcommand*\fsize{\dimexpr\f@size pt\relax}%
  \newcommand*\lineheight[1]{\fontsize{\fsize}{#1\fsize}\selectfont}%
  \ifx\svgwidth\undefined%
    \setlength{\unitlength}{1616.68395996bp}%
    \ifx\svgscale\undefined%
      \relax%
    \else%
      \setlength{\unitlength}{\unitlength * \real{\svgscale}}%
    \fi%
  \else%
    \setlength{\unitlength}{\svgwidth}%
  \fi%
  \global\let\svgwidth\undefined%
  \global\let\svgscale\undefined%
  \makeatother%
  \begin{picture}(1,0.34426767)%
    \lineheight{1}%
    \setlength\tabcolsep{0pt}%
    \put(0.23007472,0.33446403){\color[rgb]{0,0,0}\makebox(0,0)[lt]{\lineheight{1.25}\smash{\begin{tabular}[t]{l}$321$\end{tabular}}}}%
    \put(0.05193199,0.33446403){\color[rgb]{0,0,0}\makebox(0,0)[lt]{\lineheight{1.25}\smash{\begin{tabular}[t]{l}$231$\end{tabular}}}}%
    \put(0.05193215,0.17611521){\color[rgb]{0,0,0}\makebox(0,0)[lt]{\lineheight{1.25}\smash{\begin{tabular}[t]{l}$213$\end{tabular}}}}%
    \put(0.23007427,0.17611521){\color[rgb]{0,0,0}\makebox(0,0)[lt]{\lineheight{1.25}\smash{\begin{tabular}[t]{l}$312$\end{tabular}}}}%
    \put(0.05193184,0.01776639){\color[rgb]{0,0,0}\makebox(0,0)[lt]{\lineheight{1.25}\smash{\begin{tabular}[t]{l}$123$\end{tabular}}}}%
    \put(0.23007427,0.01776639){\color[rgb]{0,0,0}\makebox(0,0)[lt]{\lineheight{1.25}\smash{\begin{tabular}[t]{l}$132$\end{tabular}}}}%
    \put(0,0){\includegraphics[width=\unitlength,page=1]{trifecta.pdf}}%
  \end{picture}%
\endgroup%

%% file: coherence.pdf_tex
\begingroup%
  \makeatletter%
  \providecommand\color[2][]{%
    \errmessage{(Inkscape) Color is used for the text in Inkscape, but the package 'color.sty' is not loaded}%
    \renewcommand\color[2][]{}%
  }%
  \providecommand\transparent[1]{%
    \errmessage{(Inkscape) Transparency is used (non-zero) for the text in Inkscape, but the package 'transparent.sty' is not loaded}%
    \renewcommand\transparent[1]{}%
  }%
  \providecommand\rotatebox[2]{#2}%
  \newcommand*\fsize{\dimexpr\f@size pt\relax}%
  \newcommand*\lineheight[1]{\fontsize{\fsize}{#1\fsize}\selectfont}%
  \ifx\svgwidth\undefined%
    \setlength{\unitlength}{913.125bp}%
    \ifx\svgscale\undefined%
      \relax%
    \else%
      \setlength{\unitlength}{\unitlength * \real{\svgscale}}%
    \fi%
  \else%
    \setlength{\unitlength}{\svgwidth}%
  \fi%
  \global\let\svgwidth\undefined%
  \global\let\svgscale\undefined%
  \makeatother%
  \begin{picture}(1,0.10676304)%
    \lineheight{1}%
    \setlength\tabcolsep{0pt}%
    \put(0,0){\includegraphics[width=\unitlength,page=1]{coherence.pdf}}%
  \end{picture}%
\endgroup%

%% file: reconstruction.pdf_tex
\begingroup%
  \makeatletter%
  \providecommand\color[2][]{%
    \errmessage{(Inkscape) Color is used for the text in Inkscape, but the package 'color.sty' is not loaded}%
    \renewcommand\color[2][]{}%
  }%
  \providecommand\transparent[1]{%
    \errmessage{(Inkscape) Transparency is used (non-zero) for the text in Inkscape, but the package 'transparent.sty' is not loaded}%
    \renewcommand\transparent[1]{}%
  }%
  \providecommand\rotatebox[2]{#2}%
  \newcommand*\fsize{\dimexpr\f@size pt\relax}%
  \newcommand*\lineheight[1]{\fontsize{\fsize}{#1\fsize}\selectfont}%
  \ifx\svgwidth\undefined%
    \setlength{\unitlength}{1590.36499023bp}%
    \ifx\svgscale\undefined%
      \relax%
    \else%
      \setlength{\unitlength}{\unitlength * \real{\svgscale}}%
    \fi%
  \else%
    \setlength{\unitlength}{\svgwidth}%
  \fi%
  \global\let\svgwidth\undefined%
  \global\let\svgscale\undefined%
  \makeatother%
  \begin{picture}(1,0.68136307)%
    \lineheight{1}%
    \setlength\tabcolsep{0pt}%
    \put(0.08388931,0.56670288){\color[rgb]{0,0,0}\makebox(0,0)[lt]{\lineheight{1.25}\smash{\begin{tabular}[t]{l}$\scriptstyle{c}$\end{tabular}}}}%
    \put(0.08388931,0.6358261){\color[rgb]{0,0,0}\makebox(0,0)[lt]{\lineheight{1.25}\smash{\begin{tabular}[t]{l}$\scriptstyle{a}$\end{tabular}}}}%
    \put(0.16420911,0.63528083){\color[rgb]{0,0,0}\makebox(0,0)[lt]{\lineheight{1.25}\smash{\begin{tabular}[t]{l}$\scriptstyle{d}$\end{tabular}}}}%
    \put(0.16470465,0.56670411){\color[rgb]{0,0,0}\makebox(0,0)[lt]{\lineheight{1.25}\smash{\begin{tabular}[t]{l}$\scriptstyle{f}$\end{tabular}}}}%
    \put(0.16454408,0.60790436){\color[rgb]{0,0,0}\makebox(0,0)[lt]{\lineheight{1.25}\smash{\begin{tabular}[t]{l}$\scriptstyle{e}$\end{tabular}}}}%
    \put(0.24469378,0.63528144){\color[rgb]{0,0,0}\makebox(0,0)[lt]{\lineheight{1.25}\smash{\begin{tabular}[t]{l}$\scriptstyle{g}$\end{tabular}}}}%
    \put(0.2453545,0.56670472){\color[rgb]{0,0,0}\makebox(0,0)[lt]{\lineheight{1.25}\smash{\begin{tabular}[t]{l}$\scriptstyle{i}$\end{tabular}}}}%
    \put(0.25440835,0.60947602){\color[rgb]{0,0,0}\makebox(0,0)[lt]{\lineheight{1.25}\smash{\begin{tabular}[t]{l}$\scriptstyle{h}$\end{tabular}}}}%
    \put(0.32583947,0.63528083){\color[rgb]{0,0,0}\makebox(0,0)[lt]{\lineheight{1.25}\smash{\begin{tabular}[t]{l}$\scriptstyle{j}$\end{tabular}}}}%
    \put(0.32551341,0.60790436){\color[rgb]{0,0,0}\makebox(0,0)[lt]{\lineheight{1.25}\smash{\begin{tabular}[t]{l}$\scriptstyle{k}$\end{tabular}}}}%
    \put(0.08389423,0.60790436){\color[rgb]{0,0,0}\makebox(0,0)[lt]{\lineheight{1.25}\smash{\begin{tabular}[t]{l}$\scriptstyle{b}$\end{tabular}}}}%
    \put(0.20445145,0.62354481){\color[rgb]{0,0,0}\makebox(0,0)[lt]{\lineheight{1.25}\smash{\begin{tabular}[t]{l}$\scriptstyle{3}$\end{tabular}}}}%
    \put(0.12396708,0.62354481){\color[rgb]{0,0,0}\makebox(0,0)[lt]{\lineheight{1.25}\smash{\begin{tabular}[t]{l}$\scriptstyle{2}$\end{tabular}}}}%
    \put(0.03845151,0.61851452){\color[rgb]{0,0,0}\makebox(0,0)[lt]{\lineheight{1.25}\smash{\begin{tabular}[t]{l}$\scriptstyle{1}$\end{tabular}}}}%
    \put(0.28493581,0.62354512){\color[rgb]{0,0,0}\makebox(0,0)[lt]{\lineheight{1.25}\smash{\begin{tabular}[t]{l}$\scriptstyle{4}$\end{tabular}}}}%
    \put(0.36542078,0.62354543){\color[rgb]{0,0,0}\makebox(0,0)[lt]{\lineheight{1.25}\smash{\begin{tabular}[t]{l}$\scriptstyle{5}$\end{tabular}}}}%
    \put(0.04364729,0.41449388){\color[rgb]{0,0,0}\makebox(0,0)[lt]{\lineheight{1.25}\smash{\begin{tabular}[t]{l}$\scriptstyle{a}$\end{tabular}}}}%
    \put(0.04365189,0.38657214){\color[rgb]{0,0,0}\makebox(0,0)[lt]{\lineheight{1.25}\smash{\begin{tabular}[t]{l}$\scriptstyle{b}$\end{tabular}}}}%
    \put(0.09384183,0.39587038){\color[rgb]{0,0,0}\makebox(0,0)[lt]{\lineheight{1.25}\smash{\begin{tabular}[t]{l}$\scriptstyle{2}$\end{tabular}}}}%
    \put(0.04412747,0.24846662){\color[rgb]{0,0,0}\makebox(0,0)[lt]{\lineheight{1.25}\smash{\begin{tabular}[t]{l}$\scriptstyle{d}$\end{tabular}}}}%
    \put(0.04381769,0.2256028){\color[rgb]{0,0,0}\makebox(0,0)[lt]{\lineheight{1.25}\smash{\begin{tabular}[t]{l}$\scriptstyle{e}$\end{tabular}}}}%
    \put(0.09384183,0.23490104){\color[rgb]{0,0,0}\makebox(0,0)[lt]{\lineheight{1.25}\smash{\begin{tabular}[t]{l}$\scriptstyle{3}$\end{tabular}}}}%
    \put(0.04406423,0.124942){\color[rgb]{0,0,0}\makebox(0,0)[lt]{\lineheight{1.25}\smash{\begin{tabular}[t]{l}$\scriptstyle{f}$\end{tabular}}}}%
    \put(0.04381707,0.14511782){\color[rgb]{0,0,0}\makebox(0,0)[lt]{\lineheight{1.25}\smash{\begin{tabular}[t]{l}$\scriptstyle{e}$\end{tabular}}}}%
    \put(0.09384121,0.15441606){\color[rgb]{0,0,0}\makebox(0,0)[lt]{\lineheight{1.25}\smash{\begin{tabular}[t]{l}$\scriptstyle{3}$\end{tabular}}}}%
    \put(0.04324662,0.09208883){\color[rgb]{0,0,0}\makebox(0,0)[lt]{\lineheight{1.25}\smash{\begin{tabular}[t]{l}$\scriptstyle{g}$\end{tabular}}}}%
    \put(0.09384121,0.0739317){\color[rgb]{0,0,0}\makebox(0,0)[lt]{\lineheight{1.25}\smash{\begin{tabular}[t]{l}$\scriptstyle{4}$\end{tabular}}}}%
    \put(0.04397427,0.28666355){\color[rgb]{0,0,0}\makebox(0,0)[lt]{\lineheight{1.25}\smash{\begin{tabular}[t]{l}$\scriptstyle{c}$\end{tabular}}}}%
    \put(0.04365251,0.30608716){\color[rgb]{0,0,0}\makebox(0,0)[lt]{\lineheight{1.25}\smash{\begin{tabular}[t]{l}$\scriptstyle{b}$\end{tabular}}}}%
    \put(0.09384244,0.3153854){\color[rgb]{0,0,0}\makebox(0,0)[lt]{\lineheight{1.25}\smash{\begin{tabular}[t]{l}$\scriptstyle{2}$\end{tabular}}}}%
    \put(0.04414313,0.04527156){\color[rgb]{0,0,0}\makebox(0,0)[lt]{\lineheight{1.25}\smash{\begin{tabular}[t]{l}$\scriptstyle{i}$\end{tabular}}}}%
    \put(0.13871194,0.15108669){\color[rgb]{0,0,0}\makebox(0,0)[lt]{\lineheight{1.25}\smash{\begin{tabular}[t]{l}$\scriptstyle{\Phi_L}$\end{tabular}}}}%
    \put(0.49183956,0.64534203){\color[rgb]{0,0,0}\makebox(0,0)[lt]{\lineheight{1.25}\smash{\begin{tabular}[t]{l}$\scriptstyle{j}$\end{tabular}}}}%
    \put(0.4965441,0.54754086){\color[rgb]{0,0,0}\makebox(0,0)[lt]{\lineheight{1.25}\smash{\begin{tabular}[t]{l}$\scriptstyle{k}$\end{tabular}}}}%
    \put(0.47370116,0.44797614){\color[rgb]{0,0,0}\makebox(0,0)[lt]{\lineheight{1.25}\smash{\begin{tabular}[t]{l}$\scriptstyle{h}$\end{tabular}}}}%
    \put(0.13871194,0.07060202){\color[rgb]{0,0,0}\makebox(0,0)[lt]{\lineheight{1.25}\smash{\begin{tabular}[t]{l}$\scriptstyle{\Phi_L}$\end{tabular}}}}%
    \put(0.13871194,0.23157136){\color[rgb]{0,0,0}\makebox(0,0)[lt]{\lineheight{1.25}\smash{\begin{tabular}[t]{l}$\scriptstyle{\Phi_L}$\end{tabular}}}}%
    \put(0.13871194,0.31205603){\color[rgb]{0,0,0}\makebox(0,0)[lt]{\lineheight{1.25}\smash{\begin{tabular}[t]{l}$\scriptstyle{\Phi_L}$\end{tabular}}}}%
    \put(0.13871194,0.3925407){\color[rgb]{0,0,0}\makebox(0,0)[lt]{\lineheight{1.25}\smash{\begin{tabular}[t]{l}$\scriptstyle{\Phi_L}$\end{tabular}}}}%
    \put(0.58154296,0.07060202){\color[rgb]{0,0,0}\makebox(0,0)[lt]{\lineheight{1.25}\smash{\begin{tabular}[t]{l}$\scriptstyle{\Psi_L}$\end{tabular}}}}%
    \put(0.58154311,0.15108699){\color[rgb]{0,0,0}\makebox(0,0)[lt]{\lineheight{1.25}\smash{\begin{tabular}[t]{l}$\scriptstyle{\Psi_L}$\end{tabular}}}}%
    \put(0.5815428,0.23157136){\color[rgb]{0,0,0}\makebox(0,0)[lt]{\lineheight{1.25}\smash{\begin{tabular}[t]{l}$\scriptstyle{\Psi_L}$\end{tabular}}}}%
    \put(0.58154265,0.31205603){\color[rgb]{0,0,0}\makebox(0,0)[lt]{\lineheight{1.25}\smash{\begin{tabular}[t]{l}$\scriptstyle{\Psi_L}$\end{tabular}}}}%
    \put(0.5815428,0.392541){\color[rgb]{0,0,0}\makebox(0,0)[lt]{\lineheight{1.25}\smash{\begin{tabular}[t]{l}$\scriptstyle{\Psi_L}$\end{tabular}}}}%
    \put(0.58154311,0.47302567){\color[rgb]{0,0,0}\makebox(0,0)[lt]{\lineheight{1.25}\smash{\begin{tabular}[t]{l}$\scriptstyle{\Psi_L}$\end{tabular}}}}%
    \put(0.58154296,0.55351034){\color[rgb]{0,0,0}\makebox(0,0)[lt]{\lineheight{1.25}\smash{\begin{tabular}[t]{l}$\scriptstyle{\Psi_L}$\end{tabular}}}}%
    \put(0.58154296,0.63399501){\color[rgb]{0,0,0}\makebox(0,0)[lt]{\lineheight{1.25}\smash{\begin{tabular}[t]{l}$\scriptstyle{\Psi_L}$\end{tabular}}}}%
    \put(0,0){\includegraphics[width=\unitlength,page=1]{reconstruction.pdf}}%
  \end{picture}%
\endgroup%